\documentclass[11pt]{amsart}
\usepackage{amsmath,amssymb,amsthm,amscd}
\usepackage[a4paper,text={140mm,240mm},centering,headsep=5mm,footskip=10mm]{geometry}
\usepackage{color}
\usepackage[hypertexnames=false,linkcolor=black, colorlinks=true, citecolor=black, filecolor=black,urlcolor=black]{hyperref}
\usepackage{graphicx}
\usepackage{tikz}

\theoremstyle{plain}
\newtheorem{Thm}{Theorem}[section]
\newtheorem{Lem}[Thm]{Lemma}

\newtheorem{Cor}[Thm]{Corollary}

\theoremstyle{definition}
\newtheorem{Def}[Thm]{Definition}
\newtheorem{Rk}[Thm]{Remark}
\newtheorem{Ex}[Thm]{Example}
\newtheorem{Obs}[Thm]{Observation}
\newtheorem{Constr}[Thm]{Construction}
\newtheorem{Setup}[Thm]{Setup}

\numberwithin{equation}{section}

\newcommand{\pr}{\operatorname{pr}}
\newcommand{\gr}{\operatorname{gr}}
\newcommand{\ini}{\operatorname{in}}

\newcommand{\NP}{\operatorname{NP}}

\renewcommand{\a}{\alpha}
\renewcommand{\b}{\beta}
\newcommand{\g}{\gamma}
\renewcommand{\d}{\delta}

\renewcommand{\l}{\lambda}

\renewcommand{\r}{\rho}

\newcommand{\D}{\Delta}

\newcommand{\KK}{{\mathbb K}}

\newcommand{\QQ}{{\mathbb Q}}
\newcommand{\RR}{\mathbb R}
\newcommand{\ZZ}{{\mathbb Z}}

\newcommand{\cE}{\mathcal{E}}
\newcommand{\cF}{\mathcal{F}}
\newcommand{\cM}{\mathcal{M}}

\newcommand{\cP}{\mathcal P}

\newcommand{\Face}{\cF}
\newcommand{\Edge}{\cE}

\newcommand{\fm}{\mathfrak m}

\newcommand{\gqz}{{\geq 0}}
\newcommand{\ovl}{\overline}

\newcommand{\red}[1]{\textcolor{red}{#1}}

\begin{document}

 \author{Bernd Schober}
\address{Bernd Schober\\
Institut f\"ur Algebraische Geometrie\\
Leibniz Universit\"at Hannover\\
Welfengarten 1\\
30167 Hannover\\
Germany}
\email{schober@math.uni-hannover.de}

\title{Generalized Loose Edge Factorization Theorems}

\begin{abstract}
	We extend a factorization theorem by Gwo\'zdziewicz and Hejmej from the ring of formal power series to any complete regular local ring $ R $. 
	More precisely, let $ f \in R $
	and assume that its Newton polyhedron has a loose edge such that the initial formal of $ f $ along the latter is a product of two coprime polynomials,
	where one of them is not divided by any variable.
	Then this provides a factorization of $ f $ in $ R $. 
	As a consequence we obtain a factorization theorem for Weierstra{\ss} polynomials with coefficients in $ R $,
	which generalizes an earlier result by Rond and the author.
\end{abstract}

\subjclass[2010]{13A05, 12E05, 13F25, 14B05, 32S25}

\thanks{The author is supported by the DFG-project "Order zeta functions and resolutions of singularities" (DFG project number: 373111162)}

\maketitle

\section{Introduction}
The goal of this article is a generalization of a factorization result by Gwo\'zdziewicz and Hejmej 
from $ \KK[[x_1, \ldots, x_n]] $ to any (complete) regular local ring.
(See the results below).
We reduce their proof to the essence which is the fact that 
$ \KK[[x_1, \ldots, x_n]] $ is complete
and provide a new viewpoint via projections of the Newton polyhedron.
Note that $ R $ may even have mixed characteristics.

Let $ ( R, \fm, \KK = R/\fm ) $ be a regular local ring (not necessarily complete) with regular sytem of parameters $ (x) = (x_1, \ldots, x_n) $.
We denote by $ \widehat R $ the $ \fm $-adic completion of $ R $.
(If $ R $ is complete, we have $ \widehat R = R $).
For an element $ f \in R $, one can define the notion of a {Newton polyhedron} $ \NP(f) \subset \RR^n_\gqz $.
The latter is a closed convex set with the property 
$ \NP(f) + \RR^n_\gqz = \NP(f) $ and
coming from the set of exponents of an expansion of $ f $.
A {\em loose edge} of $ \NP(f) $ is a compact face of dimension one, say $ \Edge \subset \NP(f) $, that is not contained in any compact face of $ \NP(f) $ of dimension $ \geq 2 $. 
(In fact, we study loose edges in the slightly more general setting of $ F $-subsets later).
Associated to such an edge, we have the initial form $ \ini_\cE (f) $ of $ f $ (determined by those terms of an expansion contributing to the edge in the polyhedron)
which lies in a graded ring $ \gr_\cE (R) $ that is isomorphic to a polynomial ring $ \KK[X_1, \ldots, X_n ]$. 
(In the remaining introduction, we use this identification without mentioning).
For more details on these objects, we refer to sections \ref{sec:2} and \ref{sec:3}.

\begin{Thm}
	\label{Thm_1.1}
	Let $ R $ be a regular local ring.
	Let $ f \in R $ be an element in $ R $ such that the Newton polyhedron $ \NP(f) $ has a loose edge $ \cE $. 
	Suppose that the initial form $ \ini_\cE (f) $ of $ f $ along $ \cE $
	is a product of two coprime polynomials $ G $ and $ H $,
	where $ G $ is not divided by any variable.
	Then there exist elements $ \widehat g, \widehat h \in \widehat R $ in the completion of $ R $ such that 
	\[
		f = \widehat g \cdot \widehat h \ \ \ \mbox{ in } \widehat R
	\] 
	and $ \ini_{\cE_1} (\widehat g) = G $ and $ \ini_{\cE_2} (\widehat h) = H $, for certain faces $ \cE_1, \cE_2 $ of dimension at most one such that $ \cE= \cE_1 + \cE_2 $. 
\end{Thm}

We cannot avoid to pass to the completion since the construction of the elements $ \widehat g $ and $ \widehat h $ 
is not necessarily finite and hence may lead to an infinite series.

\begin{Cor}
	\label{Cor_1.2}
	Assume that the Newton polyhedron of $ f \in R $ has a loose edge and at least three vertices.
	Then $ f $ is not irreducible in $ \widehat R $.
\end{Cor}

\begin{Cor}
	\label{Cor_1.3}
	Assume that the Newton polyhedron of $ f \in R $ has a loose edge $ \cE $.
	If $ f $ is irreducible in $ \widehat R $, then $ \cE $ is the only compact edge of $ \NP(f) $ and 
	\[
		\ini_\cE (f) = \epsilon \cdot  P^k,
	\] 
	where $ \epsilon \in \KK^\times $ is a unit and 
	$ P \in \KK[X_1, \ldots, X_n] $ is an irreducible polynomial.
	Moreover, if the residue field $ \KK $ is algebraically closed, $ P = X^\a +  \l X^\b  $, for some $ \l \in \KK^\times $ and $ \a - \b \in \ZZ^n $ is a primitive lattice vector.
\end{Cor}

An edge $ \Edge $ is called {\em descendant} if it is parallel to a vector $ \d = (\d_1, \ldots, \d_n ) \in \RR^n $ such that
$ \d_i \geq 0 $, for $  1 \leq i \leq n - 1 $, and 
$ \d_n < 0 $.

\begin{Thm}
	\label{Thm_1.4}
	Let $ f \in R[z] $.
	Assume that the Newton polyhedron $ \NP(f) $ has a descendant, loose edge $ \cE $.
	If $ \ini_\cE (f) $ is a product of two coprime polynomials $ G, H  \in \KK[X_1, \ldots, X_n, Z]$,
	where $ G $ is monic with respect to $ Z $,
	then there exist $ \widehat g, \widehat  h \in \widehat R [z] $ such that
	\[
		f = \widehat g \cdot \widehat h \ \ \ \mbox{ in } \widehat R[z]
	\]
	where $ \widehat  g \in \widehat R [z] $ is monic, 
	$ \ini_{\cE_1} (\widehat g) = G $ and $ \ini_{\cE_2} (\widehat h) = H $,
	 for certain faces $ \cE_1, \cE_2 $ of dimension at most one such that $ \cE= \cE_1 + \cE_2 $. 
\end{Thm}

In Theorems \ref{Thm_1.1} and \ref{Thm_1.4}, 
$ \cE_1 $ is an compact edge of $ \NP(\widehat g) $ parallel to $ \cE $
and $ \cE_2 $ is either a compact edge of $ \NP(\widehat h) $ 
parallel to $ \cE$ or a vertex.

Theorems \ref{Thm_1.1}, \ref{Thm_1.4} and Corollaries \ref{Cor_1.2}, \ref{Cor_1.3} 
generalize Theorems 1.1, 1.4 and Corollaries 1.2, 1.3 of \cite{GH},
where the case $ R = \KK[[x_1, \ldots, x_n]] $ is considered.
The key step to transfer the proofs of \cite{GH} into the more general setting is to lift an element $ G \in \gr_\cE (R) $ of the graded ring to an element $ g \in R $. 
(Note that $ g $ is not unique in general).
Besides that, we provide a different perspective on the refinement of the grading of $ \gr_\cE(R) $, by considering the projection of $ \RR^n_\gqz $ along the vector $ \d \in \RR^n $ determined by the direction of the edge $ \cE$.
In particular, we show that $ \cE $ being loose implies that the projection of the Newton polyhedron along $ \d $ has exactly one vertex corresponding to $ \cE $ (Lemma~\ref{Lem:loose=>orthant}).
In contrast to \cite{GH}, we formulate convex geometry results that are used to study Newton polyhedra with loose edges in a more general variant in terms of $ F $-subsets.   

In \cite{GH} section 3, one may find other known results for which our results can be considered as generalizations. 
In particular, Theorem \ref{Thm_1.4} is some kind of generalization of a result by Rond and the author \cite{RS},
where $ R = \KK[[x_1, \ldots, x_n]] $, for any field $ \KK $
and
$ f \in R[z] $ is a Weierstra{\ss} polynomial of degree $ d $ such that 
the projection of $ \NP(f) $ along an edge containing $ (0, \ldots, 0 ,d) \in \RR^{n+1}_\gqz  $ has exactly one vertex.
This type irreducibility criterion is very useful in the study of quasi-ordinary hypersurfaces (see \cite{ACLM} or \cite{MS}).
Therefore, our main results open interesting new directions in the context of constructing Teissier's overweight deformations \cite{T} following the philosophy of \cite{MS}.

Throughout the article, we use multi-index notation:
$ x^A :=x_1^{A_1}\cdots x_n^{A_n} $ for some $ A =  (A_1, \ldots, A_n ) \in \ZZ^n_\gqz $.

\section{Newton Polyhedron and Graded Rings}
\label{sec:2}

We provide the definitions of the Newton polyhedron
and the initial form along a face of the Newton polyhedron. 
 
Let $ (R, \fm, \KK = R/\fm) $ be a regular local ring (not necessarily complete)
and
let $ (x) = (x_1, \ldots, x_n) $ be a regular system of parameters for $ R $.
We consider $ f \in R \setminus \{ 0 \} $.
Since $ R $ is Noetherian and since the map $ R \subset \widehat R $ is faithfully flat, 
$ f $ has a {\em finite} expansion 
\[
	f = \sum_{A} \r_A x^A,
	\  \ \ \
	\mbox{ for } \r_A \in R^\times \cup \{ 0 \}.
\] 
The {\em Newton polyhedron $ \NP(f) := \NP(f,x) $ of $ f $} is defined as the smallest closed convex subset of $ \RR_\gqz^n $ 
containing all points of the set 
\[
\{ A \in \ZZ_{\geq 0}^n \mid \r_A \neq 0 \}  + \RR^n_ \gqz .
\]

A linear form $  L = L_\l : \RR^n \to \RR $ is
a map defined by
\[
	L(v ) := \l_1 v_1 + \ldots + \l_n v_n = \langle \l, v \rangle, 
\]
for $ v = (v_1, \ldots v_n) \in \RR^n $ and some fixed $ \l = (\l_1, \ldots, \l_n ) \in \RR^n_\gqz $.
Given $ L $, we define
\[
	\D(L) := \{ v \in \RR^n_\gqz \mid L(v) \leq 1 \}.
\]
If $ \l \in \QQ^n_\gqz $, then $ L $ is called {\em rational}.
If $ \l \in \RR_+^n $, then we say that $ L $ is {\em positive}.

\smallskip

A closed convex subset $ \D \subset \RR^n_\gqz $ such that $ \D + \RR^n_\gqz = \D $ is called a {\em $ F $-subset of $ \RR^n_\gqz $},
see \cite{HiroCharPoly} p.~260.
We extend this notion by calling a closed convex subset $ \D \subset \RR^n $ a $ F $-subset if $ \D + \RR^n_\gqz = \D $.
Clearly, $ \NP( f )$ is an example of a $ F $-subset.

\begin{Def}
	Let $ \D \subset \RR^n_\gqz $ be a $ F $-subset of $ \RR^n_ \gqz $.
		A convex subset $ \Face \subset \D $ is called a {\em face of $ \D $} if there exists a linear form $ L $ such that 
		\[
			\D \cap \D(L) = \Face.
		\]
		If $ L $ is positive, then $ \Face $ defines a {\em compact} face.
		A {\em vertex of $ \D $} is a compact face $ v \in \D $ of dimension zero.
		An {\em edge of $ \D $} is a compact face $ \Edge \subset \D $ of dimension one.
\end{Def} 

A positive linear form $ L : \RR^n \to \RR $,
induces a monomial valuation $ \nu_L $ on $ R $ via
\[ 
	\nu_L (\r x^A) := L(A),
	\ \ \ 
	\mbox{ for } \r \in R^\times, \ A \in \ZZ^n_\gqz.
\]
For $ f = \sum_A \r_A x^A \in R \setminus \{ 0 \} $ as before, 
we have 
\[ 
	\nu_L(f) = \min \{ L(A) \mid A \in \ZZ^n_\gqz  : \r_A \in R^\times \}.
\]
  
\begin{Def}
	Let $ (R, \fm, \KK) $ be as before and let $ L : \RR^n \to \RR $ be a positive linear form.
	The {\em graded ring of $ R $ associated to $ L $} is defined as
		\[
			\gr_L(R) := \bigoplus_{ a \in \RR_\gqz }
			\cP_a / \cP_a^+, 
		\]
		where $ \cP_a := \{ f \in R \mid \nu_L (f) \geq a \} $ and 
		$ \cP_a^+ := \{ f \in R \mid \nu_L (f) > a \} $.
		
		Let $ f = \sum_A \r_A x^A \in R $ be as before.
		The {\em $ L $-initial form of $ f $} is defined as
		\[
			\ini_L(f) := \ini_L(f)_x
			:= \sum_{A : L(A) =  \nu_L(f) } \ovl{\r_A} \, X^A
			\in \cP_{ \nu_L(f)}/ \cP_{ \nu_L(f)}^+ \subset gr_L(R),
		\]
		where $ \ovl{\r_A} := \r_A \mod \fm  \in \KK $ 
		and $ (X) = (X_1, \ldots, X_n) $ denotes the images of $ (x) $ in $ \gr_L(R) $.
\end{Def}

Since $ L $ takes only values in a discrete subset of $  \RR $, 
the set $ \{ a \in \RR_\gqz \mid \cP_a / \cP_a^+ \neq 0 \} $ is a discrete subset of $ \RR $.
We observe that $ in_L(f) $ is weighted homogenous of degree $ \nu_L(f) $ with respect to the weights on $ (x) $ given by $ L $.
Since $ L $ is positive, we have 
\[
	gr_L(R) \cong \KK [X_1, \ldots, X_n].
\]

\begin{Def}
	Let $ f = \sum_A \r_A x^A \in R $ be as before.
	Let $ \Face \subset \NP(f) $ be a compact face of the Newton polyhedron
	and let $ L_\Face  : \RR^n \to \RR $ be a positive linear form determining $ \Face $.	
	The {\em initial form of $ f $ along $ \Face $} is defined as the $ L_\Face  $-initial form of $ f $,
	\[
	\ini_{\Face} (f) := \ini_{L_\Face} (f)
	\in \gr_{\Face} (R) := \gr_{L_\Face} (R) \cong \KK [X_1, \ldots, X_n].
	\]
\end{Def}	

Without loss of generality, we can choose $ L_\Face $ rational.

\smallskip 	

\section{Loose Edges and Projected Polyhedra}
\label{sec:3}

We recall the notion of a loose edge and some of their properties proven in \cite{GH}.
Furthermore, we provide a different viewpoint via a suitable projection of a given $ F $-subset.
Even though \cite{GH} considers only the case $ R = \KK[x_1, \ldots, x_n] $ the proofs apply in our more general setting since the statements are either on the convex geometry of a $ F $-subset or on the properties of the graded ring $ \gr_L(R) \cong \KK [X_1, \ldots, X_n]$,
for some positive linear form $ L $.

\begin{Def}
	Let $ \D \subset \RR^n_\gqz $ be a $ F $-subset of $ \RR^n_\gqz $.
	A {\em loose edge of $ \D $} is a compact edge $ \Edge \subset \D $ that is not contained in any compact face of $ \D $ of dimension $ \geq 2 $. 
	%
\end{Def}

\begin{Lem}
	\label{Lem:L(a)=L(b)=>L(g)>=L(a)}
	Let $ \Delta \subset \RR^n_\gqz $ be a $ F $-subset
	with a loose edge $ \Edge \subset \D $ that has ends $ \a, \b \in \RR^n_\gqz $.
	Let $ L : \RR^n \to \RR $ be a linear form such that $ L(\a) = L(\b) $. 
	For every $ \g \in \D $, we have
	$
		L(\g) \geq L(\a).
	$ 
\end{Lem}

\noindent 
The same proof as for \cite{GH} Lemma 2.1 applies.
In fact, this is also a corollary from Lemma \ref{Lem:loose=>orthant} below.

The crucial point in the previous result is that $ L $ is a linear form that is not necessarily positive, see the example below.
For positive linear forms the statement is true for any compact face of $ \D $. 

\begin{Ex}
	Let $ \D \subset \RR^3_\gqz $ be the $ F $-subset defined given by the three vertices 
	$ \a = (1,0,0), \b = (0,1,0) $, and $ \g = (0,0,1) $.
	Consider the linear form $ L : \RR^3 \to \RR $ with $ L (v_1,v_2, v_3) = v_1 + v_2 $.
	Clearly, the edge $ \Edge $ with ends $ \a $ and $ \b $ is not loose.
	We observe that
	$
		L(\a) = L(\b) = 1 > 0 = L (\g).
	$
\end{Ex}

\begin{Lem}
	\label{Lem_2.2}
	Let $ \Delta \subset \RR^n_\gqz $ be a $ F $-subset
	with a loose edge $ \Edge \subset \D $ that has ends $ \a  =(\a_1, \ldots, \a_n), \b = (\b_1, \ldots, \b_n) \in \RR^n_\gqz $.
	If 
	$
		\min\{ \a_1, \b_1 \} = \ldots = \min\{ \a_n, \b_n \} = 0,
	$
	then $ \a $ and $ \b $ are the only vertices of $ \D $.
\end{Lem}

\noindent 
The same proof as for \cite{GH} Lemma 2.2 applies.
This can also be deduced from Lemma~\ref{Lem:loose=>orthant}.

A $ F $-subset $ \D \subset \RR^m_\gqz $, $ m \in \ZZ_+ $, is called {\em orthant} if it has exactly one vertex, 
i.e., if $ \D = v + \RR^m_\gqz $, for some $ v \in \RR^m_\gqz  $.
This notion plays an important role in \cite{RS}.

\begin{Rk}
	The main result of \cite{RS} uses the associated polyhedron  
	\[ 
	\D_P := \D(P;x;z) \subset \RR^d_\gqz 
	\]
	of a Weierstra{\ss} polynomial 
	$ 
	P = z^d + \sum_{(A,b)} \r_{A,b} x^A z^b \in \KK[[x_1, \ldots, x_n]][z],
	$ 
	where $ \r_{A,b} \in \KK $.
	Here, $ \D_P \subset \RR^n_\gqz $ is the projection of the Newton polyhedron $ \NP(P) \subset \RR^{n+1}_\gqz $ from the distinguished point $ (0, \ldots, 0, d) $ onto the subspace determined by the variables $ (x_1, \ldots, x_n) $.
	In other words, $ \D_P $ is the smallest $ F $-subset containing all points of the set
	$
	\left\{ d \cdot \frac{A}{d-b} \mid \r_{A,b} \neq 0 
	\right\}. 
	$
	The interesting case in \cite{RS} is when $ \D_P $ is orthant.
	Then the unique vertex corresponds to a descendant edge of $ \NP(P) $
	(that is not necessarily loose).
	
	The idea of projecting from the distinguished point  corresponding to $ z^d $ comes from resolution of singularities and is used to provide refined information on a given singularity, see \cite{HiroCharPoly}, \cite{CPmixed2}, \cite{CSdim2}, \cite{BerndThesis}.
\end{Rk}

\begin{Setup} 
	\label{Setup} 
We fix a $ F $-subset $ \D \subset \RR^n_\gqz  $ that has a loose edge $ \Edge \subset \D $ with ends $ \a, \b \in \Edge $. 
Let $ L : \RR^n \to \RR $ be a positive linear form determining the edge $ \Edge $. 
We define
\[
	\d := \b - \a.
\]
Since $ \a = (\a_1, \ldots, \a_n ) \neq \b = (\b_1, \ldots, \b_n) $, we may assume without loss of generality 
$
\b_n < \a_n.
$
This implies $ \d_n < 0 $.
Note that $ L (\d) = L(\b) - L(\a) = 0 $ and hence there exists at least one $ i \in \{ 1, \ldots, n-1 \} $ such that $ \d_i > 0 $.

Further, let $ (R ,\fm , \KK ) $ be a regular local ring with regular system of parameters $ (x_1, \ldots, x_n) $.
Recall that we denote the images of the latter in $ \gr_L(R) $ by capital letters $ (X_1, \ldots, X_n) $
and $ \gr_L(R) \cong \KK[X_1, \ldots, X_n] $.
\end{Setup} 

We adapt the idea of projecting a $ F $-subset $ \D \subset \RR_\gqz^{n} $  in a suitable way to some $ \RR^{n-1}  $.
Our goal is to obtain a refinement of the grading
$
	 \gr_L(R) = \bigoplus_{a} \cP_a / \cP_a^+.
$
For this, we do not project from a particular point, but along the vector $ \delta $ that is defined by the difference of the ends of the fixed loose edge $ \Edge $.

\begin{Constr} [\em Projection in direction $ \d $]
	Let $ \d = (\d_1, \ldots , \d_n )\in \RR^{n} $ be {\em any} vector with $ \d_n < 0 $ and $ \d_i > 0 $, for at least one $ i \in \{ 1, \ldots, n- 1 \} $. 
	Let $ v = (v_1, \ldots, v_n) \in \RR^n_\gqz $.
	The projection of $ v $ along $ \d $ to $ \RR^{n-1}  $ is given by
	\[
		\pr_\d (v) := \left( v_1 - \frac{v_n}{\d_n} \cdot \d_1, \ \ldots, \ v_{n-1} - \frac{v_n}{\d_n} \cdot  \d_{n-1}\right) \in \RR^{n-1}.
	\]
	Note that $ v - \dfrac{v_n}{\d_n} \cdot \d = ( \pr_\d (v), 0 ) $ and $ \pr_\d (v+u) = \pr_\d (v) + \pr_\d (u) $.
	This provides a map
	\[
		\begin{array}{rcl}
			\pr_\d: \RR^{n}_\gqz & \longrightarrow & \RR^{n-1}	
			\\[3pt]
			v & \mapsto & \pr_\d (v).
		\end{array}
	\]
	For $ w \in \RR^{n-1} $, we define 
	\[
		I_{\d,w} := \pr_\d^{-1}(w) \cap \ZZ^n_\gqz = \{ v \in \ZZ^n_\gqz \mid \pr_\d (v) = w \} \subset \ZZ^n_\gqz.
	\]
	
	Given a $ F $-subset $ \D \subset \RR^n_\gqz $, we define the {\em projection of $  \D $ along $ \d $}, denoted by $\D_\d $, 
	as the smallest $ F $-subset of $ \RR^{n-1} $ containing $ \pr_\d ( \D) $,
	i.e.,
	$
		\D_\d = \pr_\d (\D) + \RR^{n-1}_\gqz. 
	$
		We define 
		\[
			\RR^{n-1}_\d := (\RR^n_{\geq 0})_\delta \subset \RR^{n-1}
		\]
\end{Constr}

\begin{Rk}
	\begin{enumerate}
		\item 
		The condition $ \d_n < 0 $ and $ \d_i > 0 $, for at least one $ i $, (up to reordering the coordinates) is equivalent to the property that the line generated by $ \d $ does intersect $ \RR^n_\gqz $ only in the origin, 
		i.e., $( \d \cdot \RR ) \cap \RR^n_\gqz = \{ 0 \} $.
		This is essential to obtain that $ I_{\d,w} $ is a finite set. 
		
		\item 
		It is possible that $ \pr_\d (v) \in \RR^{n-1} \setminus \RR^{n-1}_\gqz $.
		For example, if we consider $ \d = (-1,1,-1) $ and $ v = (0,0, a) $,
		then $ ( \pr_\d (v), 0 ) = v + a \cdot \d = (-a, a, 0) $,
		for every $ a \in \RR_\gqz $.
		We observe that $ \RR^2_\d $ has a non-compact face that is not parallel to a coordinate axis:
			
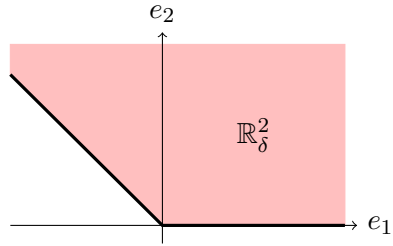
\begin{center}
				\begin{tikzpicture}[scale=0.8]
				
				\draw[fill, pink] (0,0) -- (3,0) -- (3,3) -- (-2.5,3) -- (-2.5,2.5) -- (0,0);
				
				\draw[->] (-2.5,0) -- (3.2,0) node[right]{$e_1$};;
				\draw[->] (0,-0.3) -- (0,3.2) node[above]{$e_2$};;
				
				\draw[very thick] (-2.5,2.5) -- (0,0) -- (3,0);
				
				\node at (1.5,1.5){$ \RR^2_\d $};

				\end{tikzpicture}
			\end{center} 
		
		\item 
		Suppose $ \D, \cE, \d $ are as in Setup \ref{Setup}.
		We have 
		$ \pr_\d (v) \in \RR^{n-1}_\gqz $, for all $ v \in \RR^n_\gqz $
		if and only if $ \cE $ is descendant
		(i.e., $ \d_n < 0 $ and $ \d_i \geq 0 $ for all $ i \in \{ 1, \ldots, n - 1 \} $).
		In particular, $ \RR^{n-1}_\d = \RR^{n-1}_\gqz $ in this case.
 	\end{enumerate}
\end{Rk}

The previous leads to

\begin{Def}
	Let $ \d = (\d_1, \ldots , \d_n )\in \RR^{n} $ be {any} vector with $ \d_n < 0 $ and $ \d_i > 0 $, for at least one $ i \in \{ 1, \ldots, n- 1 \} $. 
	A $ F $-subset $ \D \subset \RR^{n-1} $ is called {\em $ \d $-orthant}
	if it is of the form
	$ \D = w + \RR^{n-1}_\d + \RR^{n-1}_\gqz $,
	for a unique vertex $ w \in \RR^{n-1} $.
\end{Def}

Using this notation, we can provide a connection to \cite{RS}.

\begin{Lem}
	\label{Lem:loose=>orthant}
	Let $ \D, \cE, \d $ be as in Setup \ref{Setup}.
	Since $ \cE $ is a loose edge, we obtain that the projection $ \D_\d  \subset \RR^{n-1} $ is $ \d $-orthant.
	In particular, if $ \cE $ is descendant, then $ \D_\d \subset \RR^{n-1}_\gqz $ is orthant.
\end{Lem}

In general, the converse statement is not true,
i.e., if $ \D_\d $ is $ \d $-orthant for some vector $ \d \in \RR^n$, 
then $ \d $ does not necessarily determine a loose edge of $ \D $.

\begin{proof}
	The result follows by the same arguments as \cite{RS} Corollary 2.7 iv):
	Let $ w_1 = \pr_\d (\a) = \pr_\d (\b) $ be the vertex of $  \D_\d $ coming from the projection of the ends of $ \cE $. 
	Suppose $ \D_\d $ is not orthant.
	Then there exists at least one further vertex $ w_2 \in \D_\d $, $ w_2 \neq w_1 $, such that the segment $ [w_1, w_2] $ is contained in the boundary of $ \D_\d $.
	Hence, there exists a vertex $ \g \in \D  $ with $ \pr_\d ( \g) = w_2 $.
	Clearly, $ \a, \b , \g $ are pairwise different and the triangle defined by these three points is a face of $ \D $. 
	This contradicts the assumption that the edge given by $ \a $ and $ \b $ is loose.
\end{proof}

\begin{Obs}
	Let $ \D, \cE, L, \d, R, (x) $ be as in Setup \ref{Setup}.
	Since $ \d $ is given by $ \cE $, we have
	\[
		\gr_L(R) = \bigoplus_{ w \in \RR^{n-1}} S_{\d,w}, 
	\] 
	where $ S_{\d,w} $ is the $ \KK $-vector space with basis 
	$ 
		B_{\d,w} := \{X^A \mid A \in I_{\d,w}   \}.
	$
	Let us point out that the set $ \{ w \in \RR^{n-1} \mid S_{\d,w} \neq 0 \} \subset \RR^{n-1} $ is a discrete subset.
	Further, for all $ w \in \RR^{n-1} $ such that $ S_{\d,w} \neq 0 $, there exists at least one $ v \in \ZZ^n_\gqz $ with $ \pr_\d (v) = w $.
	
	This is compatible with
	$ 
	\gr_L(R)  = \bigoplus\limits_{a \in \RR_\gqz } \cP_a/\cP^+_a. 
	$ 
	For $ a \in \RR_\gqz $, let $ R_a := \cP_a/\cP^+_a $,
	which is the $ \KK $-vector space with basis 
	$
		B_a := \{ X^A \mid A \in \ZZ^n_\gqz \,\wedge\, L(A) = a \},
	$
	and we define
	$ I_{L,a} := \{ w \in \RR^{n-1} \mid L(w,0) = a  \} $. 
	We have $ B_a  = \bigcup\limits_{w \in I_{L,a}} B_{\d,w} $.
	Note that is a disjoint union and all but finitely many of the appearing $ B_{\d,w} $ are empty.
	Therefore,
	\[
		R_a = \bigoplus_{w \in I_{L,a}}  S_{\d,w}.
	\]
	We remark that the property $ L(\d) = 0 $ is crucial.
	The following pictures illustrates the compatibility:
	
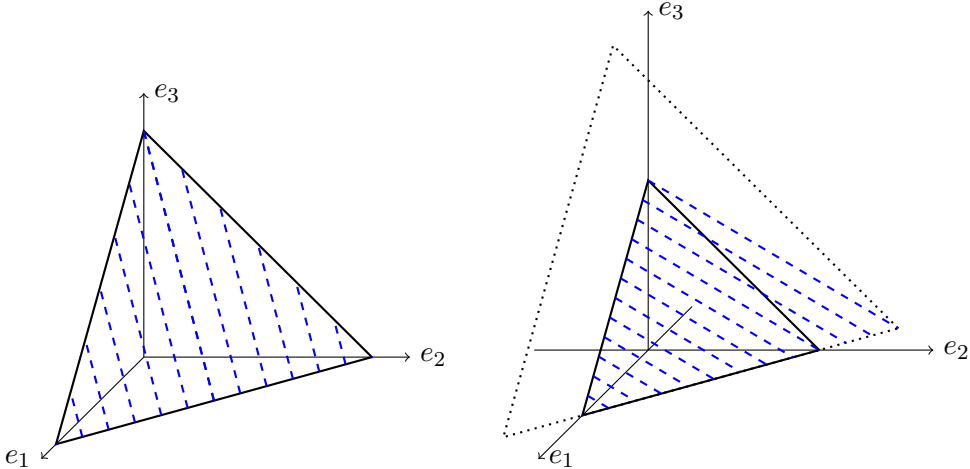
\begin{center}
		\begin{tikzpicture}[scale=1]

		\draw[->] (0,0,0) -- (3.5,0,0) node[right]{$e_2$};;
		\draw[->] (0,0,0) -- (0,3.5,0) node[right]{$e_3$};;
		\draw[->] (0,0,0) -- (0,0,3.5) node[left]{$e_1$};;
		
		\draw[thick] (3,0,0) -- (0,3,0) -- (0,0,3) -- (3,0,0);
		
		\foreach \i in {0, ..., 6}
		{
				\draw[dashed, thick, blue] (\i/4,0,3-\i/4) -- (0,\i/2,3- \i/2);
		} 
		
		\foreach \i in {6, ..., 12}
		{
			\draw[dashed, thick, blue] (\i/4,0,3-\i/4) -- (\i/2 - 3, 6 - \i/2, 0);
		} 
		\end{tikzpicture}
		\hspace{10pt}
		\begin{tikzpicture}[scale=1.5]

		\draw[->] (-1,0,0) -- (2.5,0,0) node[right]{$e_2$};;
		\draw[->] (0,0,0) -- (0,3,0) node[right]{$e_3$};;
		\draw[->] (0,0,-1) -- (0,0,2.5) node[right]{$e_1$};;
		
		\draw[thick, dotted] (2,0,-0.5) -- (-0.5,2.5,-0.5) -- (-0.5,0,2) -- (2,0,-0.5);
		\draw[thick] (1.5,0,0) -- (0,1.5,0) -- (0,0,1.5) -- (1.5,0,0);
		
		\foreach \i in {0,...,12}
		{
			\draw[dashed, thick, blue]  (0,1.5-\i/8,\i/8) -- (2-\i/6,0,\i/6-1/2);
		}
		
		\end{tikzpicture}
	\end{center} 
	The black triangle are all points $ v \in \RR^3_\gqz $ for which $ L (v) = a $, for some fixed positive linear form $ L : \RR^3 \to \RR $ and $ a \in \RR_+ $.
	The blue dashed lines show the projection lines (from $ \RR^3_\gqz $ to $ \RR^2 \times \{0 \} $) along a vector $ \d \in \RR^3 $ with $ L (\d ) = 0 $
	(with $ \d $ descendant on the left
	and $ \d $ not descendant on the right).
	The triangle on the right determined by the dotted lines is a subset of $ \{ v \in \RR^3 \mid L(v) = a \} $. 
	As we see, there are $ v \in \RR^3_\gqz $ such that $ \pr_\d (v) \in \RR^2 \setminus \RR^2_\gqz $.
\end{Obs}

In the situation of the previous observation,
we have $ L(\a ) = L (\b) $ if $ \a $ and $ \b $ denote the end points of the loose edge $ \cE $.
(Recall Lemma \ref{Lem:L(a)=L(b)=>L(g)>=L(a)}).
Furthermore, we have $ \pr_\d (\a) = \pr_\d (\b) =: u $, by construction, and hence
$
	X^\a, X^\b \in S_{\d, u}.
$

Let us point out that the constructed grading on $ \gr_L(R) \cong \KK [X_1, \ldots, X_n] $ given by 
$ \bigoplus_{ w \in \RR^{n-1}} S_{\d,w} $
is a variant of the grading $ \bigoplus_{w \in \ZZ_\gqz^{n-1}} R_w $ by \cite{GH} (which is defined by a certain weight $ \omega $,
see loc.~cit.~before Lemma 2.4).
The key in their construction is to choose a particular basis $ \xi_1, \ldots, \xi_n \in \ZZ_\gqz^{n} $ of the vector space $ \RR^n $ such that the projection along $ \d $ becomes the projection to the first this $ n - 1 $ coordinates with respect to $ \xi_1, \ldots, \xi_n $
(see loc.~cit.~Lemma 2.3).

\smallskip

The following two lemmas are the ingredients for the proof of Theorem \ref{Thm_1.1}.
For them, we need to introduce a variant of the set $ M $ defined in \cite{GH} before Lemma~2.4:
Let $ L : \RR^n \to \RR $ be a positive linear form and let $ \d \in \RR^n $ be a vector with $ \d_n < 0 $ and $ L(\d) = 0 $.
We define
\[
	\cM_\d := \pr_\d ( \ZZ^n_\gqz )  \subset \RR^{n-1}_\d.
\]
Clearly, for $ w_1, w_2 \in \cM_\d $, we have $ w_1 + w_2 \in \cM_\d $ and 
$ \dim S_{\d,u} > 0 $ implies $ u \in \cM_\d $. 

Note that $ \cM_\d \neq M $ (of \cite{GH}).
In particular, $ \dim S_{\d,w} \neq 0 $ for every $ w \in \cM $. 
 
\begin{Lem}
	\label{Lem_2.4}
	Let $  \D, \cE, L, \d, R $ be as in Setup \ref{Setup}.
	Let $ u \in \RR^{n-1} $ and $ w \in \cM_\d $.
	Assume that $ S_{\d, u} $ contains two coprime monomials.
	Then 
	\[
		\dim S_{\d,u+w} = \dim S_{\d,u} + \dim S_{\d,w} - 1.
	\]
\end{Lem}

\noindent 
The same arguments as in the proof for \cite{GH} Lemma 2.4 apply:
Using that the dimension of $ S_{\d,u} $ coincides with the number of elements in $ I_{\d,u} $ 
(analogously for $ \dim S_{\d,w} $),
the proof reduces to the combinatorial problem of determining the number of points in $ \ZZ^n_\gqz $ appearing on the sum of two parallel segments.
For more details, we refer to \cite{GH}.

The assumption that $ S_{\d, u} $ contains two coprime monomials is essential as the following example shows.
This is the reason, why we have to impose in Theorem~\ref{Thm_1.1} that $ G $ is not divided by any variable.
Another example for this (in the context of factoring a given element $ f \in R $) is given in \cite{GH} Remark 2.6.  

\begin{Ex}
	Let $ R := \KK[[x_1, x_2]] $, for any field $ \KK $.
	Consider $ \d = (3,-2) \in \RR^2 $.
	For $ w \in \RR $, we have 
	\[
		\pr_{\d}^{-1} (w) = \{ 
		v = (v_1,v_2) \in \RR^2_\gqz \mid 
		2 v_1 + 3 v_2 =  2w
		\}.
	\]
	The following picture shows
	$  \pr_{\d}^{-1} (3.5) $ (red),
	$  \pr_{\d}^{-1} (6.5) $ (blue),
	and $  \pr_{\d}^{-1} (10) $ (black),
	where filled points are lattice points corresponding to elements in $ I_{\d,w} $.
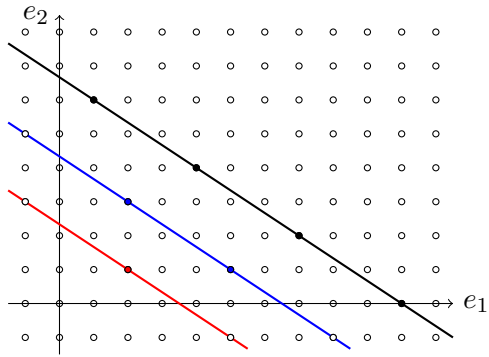
\begin{center}
	\begin{tikzpicture}[scale=0.45]

	\foreach \i in {-1,...,11}
	{
		\foreach \x in {-1,...,8}
		{
			\draw[fill = white] (\i,\x) circle [radius=0.09];	
		}
	}

	\draw[->] (-1.5,0) -- (11.5,0) node[right]{$e_1$};
	\draw[->] (0,-1.5) -- (0,8.5) node[left]{$e_2$};

	\draw[thick, red] (-3/2,10/3) -- (11/2,-4/3);
	\draw[thick, blue] (-3/2,16/3) -- (17/2,-4/3);
	\draw[thick, black] (-3/2,23/3) -- (23/2,-1);
	
	\draw[fill = white] (-1,3) circle [radius=0.09];	
	\draw[fill = white] (5,-1) circle [radius=0.09];	
	\draw[fill = red] (2,1) circle [radius=0.09];	
	\draw[fill = white] (-1,5) circle [radius=0.09];		
	\draw[fill = blue] (2,3) circle [radius=0.09];	
	\draw[fill = blue] (5,1) circle [radius=0.09];	
	\draw[fill = white] (8,-1) circle [radius=0.09];			
	\draw[fill = black] (1,6) circle [radius=0.09];	
	\draw[fill = black] (4,4) circle [radius=0.09];	
	\draw[fill = black] (7,2) circle [radius=0.09];	
	\draw[fill = black] (10,0) circle [radius=0.09];	
	
	\end{tikzpicture}
\end{center} 
	Thus,
	$ \dim S_{\d, 3.5} = \# I_{\d, 3.5} = 1 $,
	$ \dim S_{\d, 6.5} = \# I_{\d, 6.5} = 2 $,
	and $ \dim S_{\d, 10} = \# I_{\d, 10} = 4 $.
	In particular,
	$ \dim S_{\d, 3.5} + \dim S_{\d, 6.5} - 1 = 2 \neq 4  = \dim S_{\d, 10} $.
	But clearly, $ S_{\d, 6.5} $ does not contain two coprime monomials.
\end{Ex}

\begin{Lem}
	\label{Lem_2.5}
	Let $  \D, \cE, L, \d, R $ be as in Setup \ref{Setup}.
	Let $ G \in S_{\d,u} $ and $ H \in S_{\d,w} $ be coprime polynomials.
	If $ G $ is not divisible by any monomial, then
	\[
		G S_{\d,w+i} + H S_{\d,u+i} = S_{\d,u+w+i},
		\ \ \
		\mbox{for every } i \in \cM_\d. 
	\]
\end{Lem}
	
	\noindent 
	The same proof as in \cite{GH} Lemma 2.5 applies:
	The idea is to show that the sequence 
	$ 0 \to S_{\d,i} \stackrel{\Phi}{\to} S_{\d,w+i} \times  S_{\d,u+i} \stackrel{\Psi}{\to}  S_{\d,u+w+i} \to 0 $ 
	is exact,
	where $ \Phi (\eta) := (\eta H, - \eta G ) $, for $ \eta \in S_{\d,i} $,
	and $ \Psi(\psi, \varphi) := \psi G + \varphi H $, for $ (\psi, \varphi ) \in  S_{\d,w+i} \times  S_{\d,u+i} $.
	The non-trivial part is the surjectivity of $ \Psi $ which can be deduced using Lemma \ref{Lem_2.4}.
	For more details, we refer to \cite{GH}.
	
\smallskip

In order to adapt the proof of Theorem \ref{Thm_1.1} for Theorem \ref{Thm_1.4},
one needs the following two results.
	
\begin{Lem}
	\label{Lem_2.7}
	Let $  \D, \cE, L, \d, R $ be as in Setup \ref{Setup}.
	Let $ G \in S_{\d,u} $ and $ H_j \in S_{\d,w_j} $, for $ u, w_j \in \cM_\d $ and $ j \in \{ 1,2\} $.
	Assume that, for every $ i \in \cM_\d $,
	\[
		G S_{\d,w_j+i} + H_j S_{\d,u+i} = S_{\d,u+w_j+i},
		\ \ j \in \{1, 2 \}. 
	\]
	Then, we have, for every  $ i \in \cM_\d $,
	\[
	G S_{\d,w_1+w_2+i} + H_1 H_2 S_{\d,u+i} = S_{\d,u+w_1+w_2+i}. 
	\]
\end{Lem}

\noindent 
The same proof as in \cite{GH} Lemma 2.7 applies:
This is a short computation applying the hypothesis in a clever way. 
For details, we refer to \cite{GH}.
	
\begin{Lem}
	\label{Lem_2.8}
	Let $  \D, \cE, L, \d, R, (x_1, \ldots, x_n) $ be as in Setup \ref{Setup}.
	Let $ G \in S_{\d,u} $ and $ H \in S_{\d,w} $ be coprime polynomials.
	If $ G $ is monic with respect to $ X_n $, then
	\[
	G S_{\d,w+i} + H S_{\d,u+i} = S_{\d,u+w+i},
	\ \ \
	\mbox{for every } i \in \cM_\d. 
	\]
\end{Lem}

\noindent 
The same proof as in \cite{GH} Lemma 2.8 applies
(recall also the paragraph before Lemma 2.8 in \cite{GH}):
First, one proves the special case $ G = X_n $
and $ H \in S_{\d,w} \cap \KK[X_1, \ldots, X_{n-1}] $.
The rest follows then by Lemmas \ref{Lem_2.5} and \ref{Lem_2.7}.
For more details, we refer to \cite{GH}.

\begin{Rk}
	In contrast to \cite{RS}, we do not project from a distinguished point to $ \RR^{n-1} $. 
	(One candidate for such a point would be the end point $ \b $ of the loose edge).
	The reason for projecting along the vector $ \d $ given by the loose edge is to obtain an appropriate refinement of the grading of $ \gr_\cE(R) $ such that Lemma \ref{Lem_2.4} holds which is one of the key ingredients for the proofs.
	
	Let us mention that the projection of $ \NP(f) $ along $ \d $ is $ \d $-orthant 
	if and only if the projection of $ \NP(f) $ from $ \b $ to $ \RR^{n-1}_\gqz $ is orthant.
	Hence, this is another reason why Lemma \ref{Lem:loose=>orthant} yields a connection to \cite{RS}.
\end{Rk}

\smallskip 	

\section{Proofs}

We come to the proofs of the main theorems. 
The key step that allows to extend the results in \cite{GH} to any complete regular local ring is the following:

\smallskip 

Let
$ (R, \fm , \KK ) $ be a regular local ring, still not necessarily complete, with regular system of parameters $ (x) = (x_1, \ldots, x_n) $.
Let $ \D \subset \RR^n_\gqz $ be a $ F $-subset, $ \cE \subset \D $ be a loose edge, and $ L : \RR^n \to \RR $ be a positive linear form defining $ \cE $.
As in Setup \ref{Setup}, we introduce $ \d \in \RR^n $ with $ \d_n < 0 $ and $ L(\d)= 0 $.

Let $ w \in \cM_\d $ and let $ G \in S_{\d,w} \subset \gr_L(R) \cong \KK[X_1, \ldots, X_n] $.
We can write $ G $ as a finite sum
\[
	G = \sum_{A \in \ZZ^n_\gqz} \l_A X^A,
	\ \ \ \mbox{ for } \l_A \in \KK.
\]
Note that $ \l_A \neq 0 $ implies $ \pr_\d (A) = w $.
For every $ A \in \ZZ_\gqz^n $ with $ \l_A \neq 0 $, we choose $ \r_A \in R^\times $ with the property 
\[
	\r_A \equiv \l_A \mod \fm. 
\]
Otherwise, we set $ \r_A := 0 \in R $.
Using this, we define 
\[
	g := \sum_{A} \r_A x^A \in R.
\]
Clearly, the image of $ g $ in $ \gr_L(R) $ is $ G $.
(In fact, we can apply this procedure for any element in $ \gr_L(R) $, not only for those in $ S_{\d, w} $).
Note that $ g $ is not unique and, in particular, $ g $ depends on a choice of a system of representatives in $ R $ for the residue field $ \KK = R/\fm $.
On the other hand, if $ \KK \subset R $, then we can uniquely choose $ \r_A := \l_A $.

\smallskip

Using the above, we adapt the proofs of \cite{GH} to prove our results.
Even though this is straight forward, we believe it is more pedagogical to give the proofs of the theorems.
Moreover, we present a slightly different argument using Lemma \ref{Lem:loose=>orthant}.

\medskip

\begin{proof}[Proof of Theorem \ref{Thm_1.1}]
	Let $ L : \RR^n \to \RR $ be a positive linear form defining the edge $ \cE $.  
	Let 
	\[ 
		a_0 := \nu_L(f) 
		\ \ \ 
		\mbox{ and }
		\ \ \
		v := \pr_\d (\a) = \pr_\d (\b) \in \cM_\d \subset \RR^{n-1}
		,
	\] 
	where $ \a, \b \in \cE $ are the end points of the loose edge $ \cE $.
	Without loss of generality, we may assume that $ \d \in \RR^n $ fulfills the properties of Setup \ref{Setup}. 
	
	By hypothesis, we have $ \ini_L(f) = G \cdot H \in S_{\d, v}. $
	We set $ G_u := G \in S_{\d, u} $ and $ H_w := H \in S_{\d,w} $, for $ u,w \in \cM_\d $.
	Let $ g_u \in R $ (resp.~$ h_w \in R $) be a lift of $ G_u $ (resp.~$H_w $), as described before.
	We define $ \phi_1 :=  g_u \cdot h_w $, which is our first approximation of $ f $.
	For
	\[
		f_1 := f - \phi_1 = f - g_u \cdot h_w,
		\ \ \ \mbox{ we have } \ \ \
		a_1 := \nu_L(f_1) > a_0.
	\]
	The vertices of $ \NP(f_1) $ lie in $ \ZZ^n_\gqz $
	which implies that the vertices of $ \NP(f_1)_\d $ are contained in $ \cM_\d $.
	Since the projection $ \NP(f)_\d $ is $ \d $-orthant (Lemma~\ref{Lem:loose=>orthant}),
	we get that each vertex of $ \NP(f_1)_\d $ is of the form 
	\[
		v + i = u + w + i, 
		\ \ \ \mbox{ for some } i \in \cM_\d, i \neq 0 .
	\]
	In particular, $ \ini_L(f_1) \in \gr_L(R) $ can be written as
	\[
		\ini_L(f_1)  = \sum_{\substack{i \in \cM_\d\\[2pt] L(u+w+i,0) = a_1 }} F_{u+w+i}^{(1)}\, ,
		\ \ \
		\mbox{ for } F_{u+w+i}^{(1)} \in S_{\d, u+w+i}.
	\]
	By assumption $ G_u $ is not divisible by a monomial, hence, by Lemma \ref{Lem_2.5}, we have
	\[
	G_u S_{\d,w+i} + H_w S_{\d,u+i} = S_{\d,u+w+i},
	\ \ \
	\mbox{for every } i \in \cM_\d. 
	\]
	Thus, for every $ i \in \cM_\d $ with $ L(u+w+i,0) = a_1 $, 
	there are $ H_{w+i} \in S_{\d, w+i} $ and 
	$  G_{u+i} \in S_{\d, u+i} $ such that
	$
		F_{u+w+i}^{(1)} = 
		G_u H_{w+i} + H_w G_{u+i}.
	$
	We choose $ h_{w+i}, g_{u+i} \in R $, as described before and define the second approximation of $ f $ by
	\[
		\phi_2 := \Big(g_u + \sum_{\substack{i \in \cM_\d, \, i \neq 0\\[2pt] L(u+w+i,0) = a_1 }} g_{u+i} \Big)
		\Big(h_w + \sum_{\substack{i \in \cM_\d,\, i \neq 0\\[2pt] L(u+w+i,0) = a_1 }} h_{w+i}\Big)
		\in R.
	\]
	Note that $ \phi_2 = g_u h_w + \sum_{i \neq 0} (g_u h_{w+i} + h_w g_{u+i}) + \sum_{i,j \neq 0 } g_{u+i} h_{w+j} $ 
	and, by construction, $ \nu_L(g_{u+i} h_{w+j}) > a_1 $ since $ L $ is positive.
	If we define
	\[
		f_2 := f - \phi_2,
		\ \ \ \mbox{ we have } \ \ \ 
		a_2 := \nu_L(f_2) > a_1.
	\]
	We continue the construction and obtain
	\[
		\widehat g := \sum_{i\in \cM_\d } g_{u+i},
		\ \ \
		\widehat h := \sum_{i\in \cM_\d } h_{w+i} 
		\ \in \
		\widehat R 
	\]
	such that $ f = \widehat g \cdot \widehat h $ in the completion $ \widehat R $,
	as desired.
	
	Recall that $ L $ is a positive linear form on $ \RR^n $ defining the edge $ \cE $.
	Let $ \cE_1 $ (resp.~$ \cE_2 $) be the face of the Newton polyhedron of $ \widehat g $ (resp.~$ \widehat h $) determined by the same $ L $.
	We have that $ \ini_{\cE_1} (\widehat g)  = G $ and $ \ini_{\cE_1} (\widehat h) = H $,
	$ \cE = \cE_1 + \cE_2 $, 
	$
		\ini_\cE (f) = \ini_{\cE_1} (\widehat g)  \cdot \ini_{\cE_1} (\widehat h)
	$,
	and $ \cE_1 $ is parallel to $ \cE $.
	(The latter is also true for $ \cE_2 $ if it is not a vertex).
\end{proof}

\smallskip

Corollary \ref{Cor_1.2} is the same as for \cite{GH} Corollary 1.2, 
except that the reference to loc.~cit.~Lemma 2.2 has to be replaced by Lemma \ref{Lem_2.2} in the present paper. 

\smallskip

\begin{proof}[Proof of Theorem \ref{Thm_1.4}]
	We follow \cite{GH}.
	Using Lemma~\ref{Lem_2.8} instead of Lemma~\ref{Lem_2.5}
	in the proof of Theorem \ref{Thm_1.1},
	we find $ \overline g, \overline h \in \widehat{R}[[z]] $ 
	such that 
	\[
		f = \overline g \cdot \overline h,
		\ \ \ \
		\ini_{\cE_1} (\overline g) = G,
		\ \
		\ini_{\cE_2} (\overline h) = H,
		\ \ \
		\cE_1 + \cE_2 = \cE. 
	\]
	Since $ \cE $ is descendant and $ G $ is monic in $ Z $,
	the Newton polyhedron of $ \overline g $ has a vertex of the form $ (0, \ldots, 0, d) $, for some $ d \in \ZZ_+ $.
	Hence, the monomial $ \epsilon z^d $ appears 
	in an expansion of $ \overline g $,
	for some unit $ \epsilon \in \widehat{R}[[z]]^\times $.
	The Weierstra{\ss} preparation theorem 
	(\cite{Bourbaki}, Ch.~VII, \S3, no.~8, Proposition 6, p.~41)
	implies that
	there exist a unit $ u \in \widehat{R}[[z]]^\times $ and $ \widehat g \in \widehat R [z] $ such that
	\[
		\overline g = u \widehat g.
	\]
	We define $ \widehat h := u \overline h $ and obtain that
	$ f = \widehat g \cdot \widehat h $.
	Since $ f, \widehat g \in \widehat R [z] $, we also have 
	$ \widehat h \in \widehat R[z] $.
	The remaining parts of the theorem follow easily.
\end{proof}

\end{document}